\documentclass[12pt]{amsart}

\usepackage{fancyhdr}
\usepackage[matrix,arrow]{xy}

\usepackage{amssymb,latexsym,amscd,amsmath,amsthm}
\usepackage[english]{babel}
\usepackage[latin1]{inputenc}
\usepackage{multind}

\input{xy}

\xyoption{all}

\newcommand{\ds}{\displaystyle}

\newcommand{\cC}{\mathcal{C}}

\newcommand{\cG}{\mathcal{G}}

\newcommand{\cM}{\mathcal{M}}

\newcommand{\Prym}{\mathrm{Prym}}

\newcommand{\Jac}{\mathrm{Jac}}

\newcommand{\Aut}{\mathrm{Aut}}

\newcommand{\lra}{\longrightarrow}

\newcommand{\ra}{\rightarrow}
\newcommand{\ol}{\overline}

\newcommand{\abf}{\bigtriangleup_\theta}
\newcommand{\ms}{\mapsto}

\newcommand{\ul}{\underline}
\newcommand{\uT}{\underline{T}}
\newcommand{\cech}{\cC^.(\fU, \uT)}

\newcommand{\ZZ}{\mathbb{Z}}

\newcommand{\fU}{\mathfrak{U}}

\newcommand{\Id}{\mathrm{Id}}

\newcommand{\Pdon}{\Prym(\pi,\Lambda)}
\newcommand{\twinv}{{H^1(Z,\underline{T})}^W}

\theoremstyle{plain}

\newtheorem{thm}{Theorem}[section]

\newtheorem{lem}[thm]{Lemma}

\newtheorem{prop}[thm]{Proposition}

\newtheorem{cor}[thm]{Corollary}

\newtheorem{rem}[thm]{Remark}
\newtheorem{defi}[thm]{Definition}

\newtheorem{ex}{Example}

\title{$N-$Bundles for $N$ an extension of a finite group by an abelian group}

\author{Yashonidhi Pandey}
\address{Chennai Mathematical Institute, Plot No. H1, Sipcot IT Park, Padur Post Office, Siruseri 603103, India \texttt{Email: ypandey@cmi.ac.in}}
\begin{document}
\maketitle

\begin{abstract}
Let $W$ be a finite group and $T$ be an abelian group. Consider an extension $0 \ra T \ra N \ra W \ra 0$. For a smooth projective curve $X$, we give a precise description of the fiber of the quotient by $T$ map $q_T: \cM_X(N) \ra \cM_X(W)$ as a torsor over an abelian variety. We also prove a result on Mumford groups. 
\end{abstract}

\section{Introduction}
Let $T$ be an arbitrary abelian group and $W$ be an arbitrary finite group. We fix an action $\sigma : W \ra Aut(T)$. 

Let $\pi:Z \ra X$ be an \'etale Galois cover of smooth projective curves of Galois group a finite group $W$.  Let $E_T$ be a $T-$bundle on $Z$. Combining the pull-back action of $W$ on $Z$ and the action $\sigma$ on the fibers, we have the twisted action of $W$ on $T-$bundles on $Z$ defined as $(w,E) \ms w^*E \times_w T$ where $\sigma(w):T \ra T$ is the understood extension of structure group from $T$ to itself. The action of $W$ on $Z$ does not in general lift to $E_T$, however on $W-$invariant principal $T-$bundles  on $Z$, we can define a Mumford group parametrising the 2-uples $(w,\gamma)$ where $\gamma$ is an isomorphism between $E$ and $wE$ of $T-$bundles of $Z$.

The second main result of this paper is the Proposition \ref{flechec} on Mumford groups.

\begin{prop} Let $\pi: Z \ra X$ be an \'etale Galois cover of smooth projective curves with Galois group a finite group $W$. We have a homomorphism of groups

$$\begin{matrix}
c: & H^1(Z,\ul{T})^W & \ra & H^2(W,T) \\
   & E_T             & \ms & [0\ra T \ra \cG^{\sigma}(E_T) \ra W \ra 0]
\end{matrix}$$
\end{prop}

Let $\eta$ denote the class of an extension $[0 \ra T \ra N \ra W \ra 0] \in H^2(W,T)$. 

The main theorem \ref{Ndecomposition} of this paper can be described as follows.

\begin{thm} The fiber of the quotient by $T$ map $q_T: \cM_X(N) \ra \cM_X(W)$ over $\pi: Z \ra X$ is $H^1(Z,\ul{T})^W_{\eta}.$ 
\end{thm}

We give applications of these results by taking $N$ to be the dihedral group, the Weyl group of type $B_n$, $C_n$ and $D_n$ to describe the set of $N-$bundles $\cM_X(N)$ on a curve $X$. 

Much of these results were obtained to study the Weyl group, the torus and the Lie group in the context of abelianisation. We thus use suggestively the notation
$W$ for an arbitrary finite group and $T$ for a finite abelian group. 

I wish to thank my advisor Christian Pauly for his advices and help in the preperation of this paper and my thesis.

\section{Principal $N-$bundles: notation and known results}
Let $T$ be an arbitrary abelian group and $W$ be an arbitrary finite group. We fix an action $\sigma : W \ra Aut(T)$. Let $N$ be an arbitrary extension
$$0 \ra T \ra N \ra W \ra 0$$
inducing the action $\sigma$.

\begin{defi} For $w \in W$ and a $T-$bundle $E$ we denote by $E \times_w T$ the quotient
of $E \times T$ by the relation 
$$(et,t')~ (e,w(t)t').$$
and call it the extension of structure group of $E$ by $w$. We call the following map the evaluation morphism
 $$\begin{matrix} 
\nu_w: & E_T \times_w T   & \ra & E_T \\
       & \overline{(e,t)} & \ms & ew^{-1}(t)
\end{matrix}$$
\end{defi}

\begin{defi} We define the twisted action of $W$ on $T-$bundles on $Z$ 
$$\begin{matrix}
W \times H^1(Z,T) & \ra & H^1(Z,T) \\
 (w, E)           & \ms & w^*E \times_w T
\end{matrix}$$
\end{defi}
Notice that pull-back by $w$ and extension of structure group by $w$ commute with eachother. We shall denote the twisted action of $w$ on $E$ as $w.E$.

For the rest of  this section let $E_T$ denote a $T-$bundle on $Z$ invariant under the twisted action of $W$.

\begin{defi} We define for $E_T$ the Mumford group as follows
$$\cG^{\sigma}(E)= \{ (w,\gamma) | \gamma: E \ra w.E \}$$ 
We define the composition $(w_1, \gamma_1) \circ (w_2, \gamma_2)$ as $(w_1 \circ w_2, w_1(\gamma_2) \gamma_1)$ where
$w_1(\gamma_2)= w_1^* \gamma_2 \times_{w_1} \circ \gamma_1$.
\end{defi}

Since $E$ and $w.E$ are principal $T-$bundles and $T$ is abelian, so the set of isomorphisms between them is a torsor on $T$ and we have $\Aut_T(E) = T$.  Thus we can talk of a Mumford group (instead of a sheaf of groups).

\begin{defi} For $E_T \in H^1(Z,T)^W$, we define a right action of $\cG^{\sigma}(E)$ on $E_T$ as follows: for $(w,\gamma)=n \in \cG^{\sigma}(E)$, we define $\mu_n$ as the  composition of morphisms
\begin{equation} \label{Naction}
\xymatrix{
E_T \ar[r]^{\gamma~~~~} & w^* E_T \times_w T \ar[r] \ar[d] & E_T \times_w T \ar[r]^{~~~~\nu_w} \ar[d] & E_T \\
                  &                  Z \ar[r]^w      &                Z                      &
}
\end{equation}
 \end{defi}

For ease of notation we will denote the Mumford group $\cG^{\sigma}(E)$ of $E_T$ by $N$.

\begin{rem} \label{action} It follows directly (cf. \cite{lp} [Prop 6.3]) that this action extends the action of $T$ on $E_T$ and lifts the action of $W$ on $Z$ over $X$.
\end{rem}

Consider the following action:
$$\begin{matrix}
\Psi_{E_T} : & (E \times_T N) \times N  & \lra & E \times_T N \\
             & (\overline{(e,n')},n)  & \ms  & \overline{(\mu_n(e),n^{-1}n')}
\end{matrix}$$

\begin{rem} \label{lin} It is proved in \cite{lp} Prop 6.6 that $E_N := E_T \times^T N$ admits a canonical $W-$linearisation induced from $E_T$ for $T$ a torus. But the same proof works for $T$ an arbitrary abelian group. 
\end{rem}

 We denote the linearisation by $\Phi_{E_T}$. 

\section{A result on Mumford group}
\begin{prop} \label{flechec} Let $\pi: Z \ra X$ be an \'etale Galois cover of smooth projective curves with Galois group a finite group $W$. We have a homomorphism of groups

$$\begin{matrix}
c: & H^1(Z,\ul{T})^W & \ra & H^2(W,T) \\
   & E_T             & \ms & [0\ra T \ra \cG^{\sigma}(E_T) \ra W \ra 0]
\end{matrix}$$
\end{prop}
This Proposition is the Proposition 7.1 in \cite{lp}. 
\begin{proof} We have two left exact functors

\begin{enumerate} 
\item ${\Gamma}_Z$ : \{Sheaves of $W$-modules on $Z$\} $\ra$ \{$W$-modules\} 
$$\underline{A} \ra H^0(Z,\underline{A})$$
\item ${\Gamma}^W$ : \{$W$-modules\} $\ra$ \{abelian groups \} 
$M \ra M^W$ 
\end{enumerate}
We denote the composite functor as $\Gamma^W_Z = \Gamma^W \circ \Gamma_Z$ whose  $n-$th derived fonctor is calculated by the Grothendieck spectral sequence of composite of two functors

\[ E^{p,q}_2 = H^p(W, H^q(Z,\underline{A})) \Rightarrow E^{p,q} = H^{p+q}(Z;W;\underline{A})
\]

We have the following short exact sequence of low degree terms
\[ 0 \ra E_2^{1,0} \ra E^1 \ra E^{0,1}_2 \stackrel{c}{\ra} E^{2,0}_2 \ra E^2
\]

In the case when the group $W$ acts upon $T$, this sequence becomes
\[ 0 \ra H^1(W,T) \ra H^1(Z;W;\underline{T}) \ra H^1(Z,\ul{T})^W \stackrel{c}{\ra} H^2(W,T) \ra H^2(Z;W; \underline{T})
\]
 

Let us fix a principal $T$-bundle $E_T \in \twinv$. 
We associate to the extension 
\begin{equation} \label{mumgp}
 0 \ra T \ra \cG^{\sigma}(E_T) \stackrel{p}{\ra} W \ra 0 
\end{equation}
an element of $H^2(W,T)$ by taking a set theoretic section $\alpha$ of $p$.

For each $w \in W$, we get an isomorphism $\alpha(w):E_T \ra w E_T$. We identify $Aut_T(E_T)$ and $T$ as $T$ is abelian. To $\alpha$ we associate the $2$-cocyle
\begin{equation}
\begin{matrix} \label{factorset} f_{\alpha}: & W \times W & \ra & T \\
                             & (w_1, w_2) & \ms & {\alpha (w_2w_1)}^{-1} \circ w_1 \alpha(w_2) \circ \alpha(w_1)
\end{matrix}
\end{equation}

If we change $\alpha$ by $\beta$, we have $f_{\beta} - f_{\alpha} = d \theta$, where $\theta$ 
is the map 

$$\begin{matrix}
\theta :&  W &\ra & T \\
        &  w &\ms & {\beta(w)}^{-1} \circ \alpha(w).
\end{matrix}$$
Thus the  class of $f_\alpha$ in $H^2(W,T)$ is well defined. It is equal to the classs of  (\ref{mumgp}) by the natural correspondance between extensions of groups by abelian groups and $2-$cocycles.

Let us explicit the map $c$. Let $\fU$ be a  $W$-invariant covering of  $Z$ by affine open sets 
,that is for  $U \in \fU$, $w.U \in \fU$ also. The  Cech complexes gives a resolution $ \uT \ra \cech $ of the sheaf  $\uT$ on $Z$. We deduce from the resolution two short exact sequences

\begin{equation} \label{sccech} 0 \ra Z^0(\cech) \ra C^0(\fU, \uT) \ra B^1(\cech) \ra 0 
\end{equation}
\begin{equation} \label{sccoh} 0 \ra B^1(\cech) \ra Z^1(\cech) \ra H^1(\cech) \ra 0 
\end{equation}

Let us abbreviate the above groups by $Z^0$, $C^0$, $B^1$, $Z^1$ et $H^1$ respectively.

By the twisted action
$$\begin{matrix} W \times C^i(\fU, \uT) & \ra & C^i(\fU, \uT) \\
                 (w, a_{k,..,l})        & \ms & {w(a_{w^{-1}k,..,w^{-1}l})}_{k,..,l}
\end{matrix}$$ 
we take the resolution by the complexes

\begin{equation} \label{comcech}
\xymatrix{
0 \ar[r] & Z^0 \ar[r]\ar[d]         & C^0 \ar[r] \ar[d]            & B^1 \ar[r] \ar[d]         & 0 \\
0 \ar[r] & C^.(W,Z^0) \ar[r]        & C^.(W, C^0) \ar[r]           & C^.(W, B^1) \ar[r]        & 0 
}
\end{equation}

\begin{equation} \label{comcoh}
\xymatrix{
0 \ar[r] & B^1 \ar[r] \ar[d]        & Z^1 \ar[r] \ar[d]        & H^1 \ar[r] \ar[d]        & 0 \\
0 \ar[r] & C^.(W,B^1) \ar[r]        & C^.(W,Z^1) \ar[r]        & C^.(W,H^1) \ar[r]        & 0
} 
\end{equation}

By the definition of the spectral sequence, the map 
$c$ is the composition of 
$$ H^0(C^.(W,H^1)) \ra H^1( C^.(W,B^1)) \ra H^2(C^.(W,Z^0)),$$
where the morphisms between the groups are connection morphisms.

For $E_T \in \twinv \subset  C^0(W,H^1)$, let $a:= (a_{i,j}) \in Z^1(\cech)$ be a $1$-cocycle 
antecedent of $E_T$. Associated to $a$, there is a $canonical$ choice of a $1$-cocycle $w a$  antecedent of the  principal $w E_T$ bundle for the twisted action of $W$ on $E_T$. Recall that by 
definition
$(wa)_{i,j} := {w(a_{w^{-1}i,w^{-1}j})}_{i,j}$. We consider the first terms of \ref{comcoh} and denote the vertical arrows by $d$
\begin{equation}
\xymatrix{
         & 0 \ar[d]                    & 0 \ar[d]                   & 0 \ar[d]                   &   \\
0 \ar[r] & B^1(\cech)^W \ar[r]  \ar[d] & Z^1(\cech)^W \ar[r] \ar[d] & H^1(\cech)^W \ar[r] \ar[d] & 0 \\
0 \ar[r] & C^0(W,B^1) \ar[r] \ar[d]    & C^0(W,Z^1) \ar[r] \ar[d]   & C^0(W,H^1) \ar[r] \ar[d]   & 0 \\
0 \ar[r] & C^1(W,B^1) \ar[r] \ar[d]    & C^1(W,Z^1) \ar[r]          & C^1(W,H^1) \ar[r]          & 0 \\
0 \ar[r] & C^2(W,B^1)                  &                            &                            &  
}
\end{equation}

Let us denote $d(a)$ by $h$. Since $d(E_T) = 0$, we have $h \in C^1(W,B^1)$. Explicitly $h$ is given by   
\begin{equation}
\begin{matrix} \label{d(a)} h= d(a) :  & W & \ra & Z^1(\cech) &       \\
                         & w & \ms & w(a) \circ a^{-1}  
\end{matrix}
\end{equation}
where over $u \in U_i \cap U_j$, we define $w(a) \circ a^{-1}(i,j)(u) := w(a)(i,j)(u) a(i,j)(u)^{-1}$.
We consider the first terms of (\ref{comcech}) and let us denote the vertical arrows by $d$
\begin{equation}
\xymatrix{
         & 0 \ar[d]                    & 0 \ar[d]                   & 0 \ar[d]                   &   \\
0 \ar[r] & Z^0(\cech)^W \ar[r]  \ar[d] & C^0(\cech)^W \ar[r] \ar[d] & B^1(\cech)^W \ar[r] \ar[d] & 0 \\
0 \ar[r] & C^0(W,Z^0) \ar[r] \ar[d]    & C^0(W,C^0) \ar[r] \ar[d]   & C^0(W,B^1) \ar[r] \ar[d]   & 0 \\
0 \ar[r] & C^1(W,Z^0) \ar[r] \ar[d]    & C^1(W,C^0) \ar[r] \ar[d]   & C^1(W,B^1) \ar[r] \ar[d]   & 0 \\
0 \ar[r] & C^2(W,Z^0) \ar[r]           & C^2(W,C^0) \ar[r]          & C^2(W,B^1) \ar[r]          & 0 
}
\end{equation}

We denote by $g \in C^1(W,C^0)$ an antecedent of $h$. Then for all $w \in W$, we have 
\begin{equation} \label{trivialisation}
g(w)(i) g(w)(j)^{-1} = h(w)(i,j)
\end{equation} 

Since $Aut_T(E_T) = T$ is commutative, by (\ref{trivialisation}) and (\ref{d(a)}) 
we deduce the equality

\[ w(a)(i,j) = g(w)(i) a(i,j) g(w)(j)^{-1}
\]
which means that $g$ explicitly gives an isomorphism $g(w) : E_T \ra w E_T$ in local coordinates
for all $w \in W$. Thus $g$ is a section of $p$ of the short exact sequence (\ref{mumgp}). 
Since $0 = d(d(a)) = d(h)$ so  $k:= d(g) \in C^2(W,Z^0)$ is a  $2$-cocycle. Now since $H^0(\cech)= \uT$, we have 
$$\begin{matrix} k: W \times W &  \ra & T \\
                      (w_1,w_2) &  \ms & w_1(g(w_2)) g(w_1 w_2)^{-1} g(w_1) 
\end{matrix}$$

The map $c$ sends $E_T$ to the class of $k$ in $H^2(W,T)$, by the definition of the spectral sequence. This definition is independant of the choice of the antecedant  $a$ and $g$ of $E_T$ and $h$ respectively. Since $g(w) : E_T \ra w E_T$ is an isomorphism,  $f_{g}$ (ref \ref{factorset}) is equal to  $k$ because the passage from $g$ to $k$ and to $f_{g}$ is the same. Thus a fortiori 
the image of $k$ and of $f_{g}$- the class of Mumford group in $H^2(W,T)$ coincide.
\end{proof}

\section{The fiber of $\cM_X(N) \ra \cM_X(W)$}

Fix an extension class $\eta= [0 \ra T \ra N \ra W] \in H^2(W,T)$. 
\begin{defi} One calls the abelianisation map, the map that for a $T-$bundle $E_T \in H^1(Z,\ul{T})^W_{\eta}$ associates the quotient $F_N$ of $E_N$ by $\Phi_{E_T}$. 
\end{defi}

\begin{lem} \label{quotparTsurZ} The principal bundle quotient of  $F_N$by $T$ is canonically isomorphic to $Z$.
\end{lem}
\begin{proof} To show that $F_N/T$ is canonically isomorphic to $Z$ it suffices to show that there is a canonical isomorphism of principal $W$-bundles on $Z$ between $\pi^*(F_N)$ and $\pi^*(Z)$ 
which, moreover, respects the $W$-linearisations induced by pull-back. To verify this condition, we shall firstly describe a canonical isomorphism between each principal bundle and $Z \times W$, and then show that, by transport of structure via these isomorphisms, we get the same $W$-linearisation on $Z \times W$, namely, the canonical $W$-linearisation.

Since $\pi:Z \ra X$ is a Galois covering of Galois group $W$, we have canonical isomorphisms  $$\pi^*Z \simeq Z \times_X Z \simeq Z \times W.$$ The $W$-linearisation on $\pi^*Z$ induces the canonical $W$-linearisation on $Z \times W$, that is 
$$\begin{matrix} \phi : & W \times Z \times W & \ra & Z \times W \\
                        & (w, (z,w'))         & \ms & (zw,w^{-1}w').
\end{matrix}$$

We have the natural isomorphisms 
$${\pi}^*(F_N/T) \ra \pi^*(F_N) /T \ra (E_T \times^T N)/T.$$ By the isomorphism $E_T \times_w^T T/T \simeq Z$ we deduce the following commutatif diagram where all arrows are isomorphisms

\begin{equation} \label{isopull}
\xymatrix{
 (E_T \times^T N)/T \ar[r] & E_T/T \times N/T \ar[r] & Z \times W \\
 (E_T \times_w^T T \times^T N)/T \ar[r] \ar[u]^{{(\nu_w \times^T N)}/T} & (E_T \times_w^T T)/T \times N/T \ar[r] \ar[u]^{\nu_w/T \times \Id} & Z \times W \ar[u]^{\Id}
}
\end{equation}

If we extend the upper sequence of (\ref{Naction}) by $\times^T N$ and then we quotient by the group $T$, by the isomorphism $$E \times_w T /T \simeq Z$$  and by (\ref{isopull}) we obtain the following action of $W$ on $Z \times W$ : for $(w,\gamma') \in \cG^{\sigma}(Z \times W)$

\begin{equation*} \label{Waction}
\xymatrix{
Z \times W \ar[r]^{\gamma{}'~~~} & (w^* Z \times W) \ar[r] \ar[d] & Z \times W \ar[r]^{\simeq}_{\Id} \ar[d] & Z \times W \\
                            &               Z \ar[r]^w       &                         Z               &
}
\end{equation*}
where $\gamma' : Z \ra w^* Z$ is an isomorphism. We remark that this action is constant on the second factor $W$ of $Z \times W$.
Thus the action in propositon \ref{lin} by passing to quotient by $T$ becomes $$\overline{\phi(E_T)}(w)(z,w') =(zw,w^{-1}w'): $$ which is the canonical $W$-linearisation.

\end{proof}

\begin{prop} \label{pushdown=abl} Let $\pi: Z \ra X$ be a Galois cover with Galois group $W$. Let
 $E$ be a principal $T-$bundle on $Z$. Then $E$ is a principal $\cG^{\sigma}(E)-$bundle on $X$.
\end{prop}
\begin{proof} Let us denote the Mumford group $\cG^{\sigma}(E)$ by $N$.

Consider the natural map of sheaves on $Z$
$$\begin{matrix}
\phi: & E & \ra & E \times^T N \\
      & e & \ms & \ol{(e,1)}
\end{matrix}$$
The bundle $E \times^T  N$ admits a canonical $W-$linearisation by the proposition \ref{lin} and we denote by $p: E \times^T N \ra E \times^T N/W =: F_N$ the quotient by $W$. The lemma \ref{quotparTsurZ} implies that $F_N/T$ is isomorphic to $Z$. Thus $F_N$ seen as a sheaf on $Z$ is a principal $T-$bundle. Thus the composition of arrows $p \circ \phi : E \ra F_N$ is a map of sheaves between principal $T-$bundles on  $Z$. Now $p \circ \phi$ is also $T-$linear, thus it is also a map of principal $T-$bundles. So it is an isomorphism since in the category of principal group bundles, morphisms are isomorphisms. Thus $E$ is a fortiori isomorphic to $F_N$ on $X$. Now $F_N$ is a principal$N-$bundle, and therefore so is $E$.
\end{proof}

\begin{lem} \label{mumgpE_N} Let $N$ be an arbitrary extension of $W$ by $T$ and $\eta \in H^2(W,T)$ denote its class. Let $p: E_N \ra X$ be a principal $N-$bundle on $X$ and we denote by  $\pi: Z \ra X$ its quotient by $T$. There exists a canonical isomorphism between $N$ and the Mumford group of the principal bundle $E_N$ seen as a principal $T-$bundle on $Z$. Moreover via these isomorphisms the natural actions of these groups on $E_N$ can also be identified.
\end{lem}
\begin{proof} For $n \in N$, let $\overline{n}$ denote the class of $n$ in $W$. For all $n \in N$, for the action of $N$ on the princiapl $N-$bundle $E_N$ and of $W$ on the principal $W-$bundle $Z$, we have the following cartesian  diagram preserving the  fibers of $E_N \ra Z \ra X$
\begin{equation*}
\xymatrix{
E_N \ar[r]^n \ar[d]          & E_N \ar[d] \\
Z   \ar[r]^{\overline{n}}    & Z. 
}
\end{equation*}
Now $E_N$ is a principal $T-$bundle on $Z$. The commutation of the above diagram is equivalent to the fact that $E_N$ is a $W-$invariant principal $T-$bundle for the twisted action of the group $W$. Thus, we conclude that the Mumford group of  $E_N$, seen as a principal $T-$bundle on $Z$, is $N$. From this it also follows that the action of the Mumford group upon $E_N$ gets identified with the usual  multiplication action by $N$.
\end{proof}

\begin{lem} \label{recMXN} Let $p: E_N \ra X$ be a principal $N-$bundle on $X$. We denote by  $\pi: Z \ra X$ the quotient by $T$. The bundle $E_N$ seen as a principal $T-$bundle on $Z$ is sent by the abelianisation map to  $E_N$ seen as a principal $N-$bundle on  $X$.
\end{lem}

\begin{proof} By the lemma \ref{mumgpE_N} we conclude that $E_N \in H^1(Z,\underline{T})^W$ and since its Mumford group is $N$ so $E_N \in H^1(Z, \underline{T})^W_{\eta}$. Since $E_N$ is a principal $N-$bundle on $X$, we have a canonical isomorphism 
$$\begin{matrix}
E_N \times N & \ra & E_N \times_X E_N \\
  (e,n)      & \ms & (e,en).
\end{matrix}$$
When we quotient $E_N \times N$ by $T$ by putting the relation $(et,n)=(e,tn)$, the last isomorphism implies the relation  $(et,etn)= (e,etn)$ on $E_N \times_X E_N$ and we obtain 
$$\begin{matrix}
 \alpha: & E_N \times_T N    & \ra & Z \times_X E_N \\
         & \overline{(e,n)}  & \ms & (\overline{e},en)
\end{matrix}$$
where $\overline{e}$ is the image of $e \in E_N$ in $Z$.
Recall the action of $N$ on $E_N \times_T N$ of abelianisation 
$$\begin{matrix}
N \times E_N \times_T N & \ra & E_N \times_T N \\
    (n,(e,n'))              & \ms & (e.n,n^{-1}n'),
\end{matrix}$$

and consider the action of $N$ on $Z \times_X E_N$ 
$$\begin{matrix}
N \times  Z \times_X E_N & \ra & Z \times_X E_N \\
         (n,(z,e))            & \ra & (z\overline{n},e). 
\end{matrix}$$

Now for these actions for all $n \in N$, we have the following commutative diagram
\begin{equation*}
\xymatrix{
E_N \times_T N \ar[d]^{\alpha} \ar[r]^{n}                 & E_N \times_T N \ar[d]^{\alpha} \\
Z \times_X E_N   \ar[r]^{n}                                    & Z \times_W E_N.
}
\end{equation*}
Thus, $E_N$ - the quotient of $Z \times_X E_N$ by $W$ is isomorphic on $X$ to the quotient of $E_N \times_T N$ by $W$ for the canonical $W-$linearisation of abelianisation. This last quotient is the image of $E_N$ seen as a principal $T-$bundle on $Z$ by $\Delta_\theta$. Now the assertion follows.
\end{proof}

\begin{cor} \label{ablinj} Let $\pi: Z \ra X$ be a Galois \'etale covering. Let $p: E \ra Z$ be a principal $T-$bundle. The abelianisation map $$\Delta_\theta: H^1(Z,\ul{T})^W_{\eta} \ra \cM_X(N)$$ is injective.
\end{cor}
\begin{proof} 
By the proposition \ref{pushdown=abl}, $E$ is a principal $N-$bundle on $X$, where $N$ is the Mumford  group. By the lemma
\ref{recMXN}, we have $\abf(E) = E$ seen as a principal $N-$bundle on  $X$ . If $E$ and $F$ are     principal $T-$bundles on $Z$ such that $\abf(E)=\abf(F)$ then $E=F$ as principal $N-$bundles on $X$. By the lemma \ref{quotparTsurZ} we have $E/T$ (resp. $F/T$) is isomorphic to  $Z$. Thus $E$ and $F$ are isomorphic as principal $T-$bundles on $Z$.
\end{proof}

\begin{thm} \label{Ndecomposition} The fiber of the quotient by $T$ map $q_T: \cM_X(N) \ra \cM_X(W)$ over $\pi: Z \ra X$ is $H^1(Z,\ul{T})^W_{\eta}.$ 
\end{thm}
\begin{proof} By the lemma \ref{quotparTsurZ} the map quotient by $T$ maps  $H^1(Z,\uT)^W_{\eta}$ to the fiber of $q_T$. By the lemma \ref{ablinj}, this map to the fiber of $q_T$ is injective and by \ref{recMXN} it is surjective.  
\end{proof}

\section{Applications}

\begin{ex}
Let $\pi: Z \ra X$ be a covering Galois group $D_{2n}$ the dihedral group of order $2n$. When we quotient $Z$ by $\ZZ/n\ZZ \subset D_{2n}$, we obtain a double covering $p: Y \ra X$. We have the following exact sequence 
$$ 0 \ra \ZZ/n\ZZ \ra D_{2n} \ra \ZZ/2\ZZ \ra 0$$ with the action of $\ZZ/2\ZZ$ on $\ZZ/n\ZZ$ by $e(\ol{k})=\ol{k}$ and $\sigma(\ol{k})=-\ol{k}$. The curve $Z$ determines a primitive element, denoted as $Z$ again, in  $\Jac(Y)[n]$ invariant under the action of $\ZZ/2\ZZ$ switching the two sheets. 

Let us suppose that $n$ is odd. We denote by $\Prym(Y/X)$ the Prym variety associated to the \'etale double cover $p: Y \ra X$. We have the isogeny $\Jac(X) \times \Prym(Y/X) \ra \Jac(Y)$. 
Let us consider the points of $n-$torsion in $\Jac(Y)$, denoted $\Jac(Y)[n]$. 
Thus the isogeny gives an isomorphism between abelian groups

$$ \Jac(X)[n] \times \Prym(Y/X)[n] \ra \Jac(Y)[n],$$
since the kernel of the isogeny is consists of points of $2-$torsion in $\Jac(X)$. Now the involution $\sigma: Y \ra Y$ above $X$ operates as $+\Id$ on $\Jac(X)[n]$ and  $-\Id$ on $\Prym(Y/X)[n]$. Let $\alpha= (\beta, \gamma) \in \Jac(X)[n] \times \Prym(Y/X)[n]$.
\begin{enumerate}
\item If $\sigma(\alpha)= \alpha$, then $(\beta, - \gamma) = (\beta, \gamma)$, that is  $\gamma=0$. Thus $\alpha \in \Jac(X)[n]$. Thus the \'etale covering  $\alpha$ is the pull-back on  $Y$ of a cyclic \'etale covering of degree $n$ on $X$ and the associated total covering $\alpha$ is Galois of Galois group  $\ZZ/2\ZZ \times \ZZ/n\ZZ$.
\item If $\sigma(\alpha) = - \alpha$, then $(\beta, -\gamma) = (-\beta, - \gamma)$, that is  $\beta = 0$. Thus  $\alpha \in \Prym(Y/X)[n]$. Thus $\alpha$ is a cyclic \'etale covering of order $n$ on $Y$. The total covering $$Z \stackrel{n:1}{\ra} Y \stackrel{2:1}{\ra} X$$ 
gives an \'etale Galois covering of Galois group $D_{2n}$.
\end{enumerate}
\end{ex}

\begin{prop} \label{Wactiont} The Weyl group $W$ acts transitively on the fibers  of the map  $$\cM_X(T) \ra \cM_X(N)$$ of extension of structure group from $T$ to $N(T)$. The action is generically without fixed points. 
\end{prop}
\begin{proof} For $E \in H^1(X,\ul{T})$, let us denote $E \times^T N$ by $E_N$. We consider the short exact sequence of group schemes 

\begin{equation}
0 \ra \Aut_T(E) \ra \Aut_N(E_N) \ra \underline{W} \ra 0 
\end{equation}
We have the associated long exact sequence (we omit the curve $X$ in the notation)
$$\begin{matrix}
0 & \ra & H^0(\Aut_T(E)) & \ra & H^0(\Aut_N(E_N)) & \ra & H^0(\underline{W}) \\   & \stackrel{\delta}{\ra} &
H^1(\Aut_T(E)) & \ra & H^1(\Aut_N(E_N)) & \ra & H^1(\underline{W}) 
\end{matrix}$$

The distinguished elements of the sets $H^1(\Aut_T(E))$ and $H^1(\Aut_N(E_N))$ are $E$ and $E_N$ respectively. Let $\sigma : X \ra E_N/T$ be the section corresponding to $E$. The group $H^0(\underline{W})=W$ acts on $E$ in the following way: $\delta(w) E$ is the principal $T$-bundle obtained  by the section $(\sigma,w): X \ra E_N \times^N N/T \simeq E_N/T$. Thus we have 

\begin{equation} \label{weylactionTfibre}
\delta(w) E = E \times^w T.
\end{equation}
 
As the sequence gives the fiber on the distinguished element, we deduce that the group $W$ acts transitively on the fiber of $E_N$, from which the first assertion follows. For a generic $E \in H^1(X,T)$ we have $Stab(E)=\{e\}$ which implies the second assertion.

\end{proof}

\begin{prop} Let $T$ be a torus of a Lie group. Let $0 \ra T \ra N \ra W \ra 0$ be an extension of $W$ by $T$. The map extension of structure group from $N$ to $W$ $$\cM_X(N(T)) \ra \cM_X(W)$$ is surjective.
\end{prop}
\begin{proof} An element $ (\pi: Z \ra X) \in \cM_X(W)$ corresponds to an \'etale Galois covering of curves on $X$. Since  $X$ is smooth and $\pi$ is \'etale, so $Z$ is smooth. Since $X$ is projective and $\pi$ is finite, so $Z$ is projective also. By the lemma 7.2 of Lange-Pauly \cite{lp}, we have that the set $H^1(Z,\ul{T})^W_{\eta}$ is non-empty where $\eta$ denotes the extension  class of $[0 \ra T \ra N \ra W \ra 0] = \eta \in H^2(W,T)$. By the lemma \ref{quotparTsurZ} $\pi: Z \ra X$ is the image of the composition of arrows $$\Pdon_{\eta} = H^1(Z, \underline{T})^W_\eta  \ra H^1(X,\underline{N}) \ra H^1(X, \underline{W}).$$
\end{proof}

\begin{ex} \label{WfibreBC} Let us consider a Weyl group $W$ of type $B_n$ or $C_n$. It is isomorphic to $(\ZZ/2 \ZZ)^n \ltimes \Sigma_n$. By the theorem \ref{Ndecomposition}, we have
$$\cM_X(W) = \ds{\bigsqcup_{(\pi: Z \ra X) \in \cM_X(\Sigma_n)} H^1(Z, \ul{T})^{\Sigma_n}_{\eta}}$$
where $T= (\ZZ/2\ZZ)^n$ and $\eta$ is the extension class  
$$[0 \ra (\ZZ/2\ZZ)^n \ra W \ra \Sigma_n \ra 0] = \eta \in H^2(\Sigma_n, (\ZZ/2\ZZ)^n).$$
Now a principal $T-$bundle on  $Z$ is a $n-$uple $(L_1, \cdots, L_n)$ where $L_i \in \Jac(Z)[2]$ are line bundles of order $2$.By the equation \ref{weylactionTfibre} of the proposition \ref{Wactiont},
we deduce that $\Sigma_n$ operates by permuting factors. Thus, the $L_i$ are isomorphic to eachother. Thus we have
$$H^1(Z,\ul{T})^W_{\eta} = \{L \in \Jac(Z)[2]| [\cG^{\sigma}((L, \cdots, L))] = \eta \}$$
\end{ex}

\begin{ex} The Weyl group $D_n$ is isomorphic to $(\ZZ/2\ZZ)^{n-1} \ltimes \Sigma_n$ where $\Sigma_n$ acts by permuting the factors on the subgroup of $(\ZZ/2\ZZ)^n$ having an even number of ones. Reasoning as before in the example \ref{WfibreBC} we find that  
$$ \cM_X(W) = \left\{ \begin{array}{rcl}
\cM_X(\Sigma_n) & n \quad \textrm{odd} \\
\ds{\bigsqcup_{(\pi: Z \ra X) \in \cM_X(\Sigma_n)}  \{L \in \Jac(Z)[2]| [\cG^{\sigma}((L, \cdots, L))] = \eta \}} & n \quad \textrm{even}
\end{array}
\right.$$
\end{ex}

\bibliographystyle{amsplain}

\end{document}